\documentclass[12pt,reqno]{amsart}

\usepackage[T1]{fontenc}
\usepackage{microtype,textcomp,fbb}
\usepackage[a4paper,margin=1in]{geometry}

\usepackage{
    amsmath, amssymb, amsthm, amscd,
    gensymb, graphicx, etoolbox,
    booktabs, stackrel, mathtools
}
\usepackage[colorlinks=true, linkcolor=blue, citecolor=blue, urlcolor=blue, breaklinks=true]{hyperref}
\usepackage[capitalise]{cleveref}
\usepackage{placeins}
\newtheorem{theorem}{Theorem}[section]

\newtheorem{lemma}[theorem]{Lemma}

\theoremstyle{definition}
\newtheorem{definition}[theorem]{Definition}
\newtheorem{example}[theorem]{Example}
\newtheorem{question}[theorem]{Question}

\usepackage{tikz,float}
\usepackage{caption}
\usepackage{subcaption}
\usepackage{comment}
\usetikzlibrary{positioning, shapes, fit, arrows}
\usepackage{tkz-graph}
\tikzstyle{vertex}=[circle, draw, inner sep=0pt, minimum size=6pt]

\def\Ind{\mathrm{Ind}}
\def\del{\mathrm{del}}
\def\lk{\mathrm{lk}}
\def\Ch{\mathrm{Ch}}

\begin{document}

\title{On the Vertex Decomposability of $r$-Independence Complexes of Trees}

\author{Rutuja Sawant}
\address{Chennai Mathematical Institute, Siruseri, Tamil Nadu, India}
\email{rutuja@cmi.ac.in}


\vspace*{-0.4cm}
\begin{abstract}
Let $G$ be a graph and $r \ge 1$. A vertex subset is $r$-independent if every connected component of its induced subgraph has size at most $r$. The family of all such subsets forms a simplicial complex, the $r$-independence complex $\Ind_r(G)$, generalizing the classical independence complex.

Recent work has focused on shellability and vertex decomposability of these complexes. For chordal graphs, $\Ind_r(G)$ has the homotopy type of a wedge of spheres for all $r$, and some chordal subfamilies are known where these complexes are not even sequentially Cohen–Macaulay. Thus, determining chordal graph classes and values of $r$ for which $\Ind_r(G)$ is sequentially Cohen–Macaulay, shellable, or vertex decomposable remains an active area. Existing methods, based on chordal hypergraphs or special graph properties, do not extend to arbitrary chordal graphs.

In this paper, we show that for every tree $T$ and every integer $r \ge 1$, the complex $\Ind_r(T)$ is vertex decomposable, resolving a conjecture \cite[Conjecture 3.15]{PD23chordal} of Abdelmalek et al. Our approach gives a structural description of shedding vertices via rooted subtrees and uses it to prove vertex decomposability recursively.
\end{abstract}

\maketitle

\section{Introduction}

Independence complexes have long formed an important interface between graph theory, combinatorial topology, and commutative algebra, with deep connections to Stanley-Reisner theory and related algebraic invariants \cite{EB07,AD09,AE11,FC05,HH04,SM16,RV90,MW97,RW09}. A natural generalization, known as the $r$-independence complex, was introduced by Paolini and Salvetti \cite{PG17}. For an integer $r\ge1$, a subset $A\subseteq V(G)$ is called \emph{$r$-independent} if every connected component of the induced subgraph $G[A]$ has size at most $r$. The collection of all such subsets forms a simplicial complex, called the \emph{$r$-independence complex} of $G$, denoted by $\Ind_r(G)$. When $r=1$, this construction recovers the classical independence complex. Higher independence complexes have since attracted considerable interest because of their connections with domination parameters, homological connectivity, and hypergraph methods \cite{PD22,PD21,AG25}.

A central theme in the study of independence-type complexes is to understand when they satisfy strong combinatorial properties such as shellability and vertex decomposability. These notions are related through the implications
\[
\text{Vertex decomposable}
\Rightarrow
\text{Shellable}
\Rightarrow
\text{Sequentially Cohen-Macaulay},
\]
see \cite{JJonsson08}. In particular, shellable complexes have the homotopy type of a wedge of spheres, although the converse need not hold. The notion of vertex decomposability was introduced by Provan and Billera \cite{PB80} and later extended to non-pure complexes by Bj\"orner and Wachs \cite{BW96,BW97}. In the classical case $r=1$, Dochtermann and Engstr\"om \cite{AD09} and Woodroofe \cite{RW09} proved that if $G$ is chordal, then $\Ind(G)$ is vertex decomposable. These results naturally motivate the investigation of analogous structural properties for the higher independence complexes $\Ind_r(G)$.

Motivated by this, Abdelmalek et al.~\cite{PD23chordal} studied shellability and vertex decomposability of higher independence complexes through methods from the theory of chordal hypergraphs. Using a hypergraph associated to a graph $G$, they claimed that $\Ind_r(T)$ is shellable for every tree $T$, and proved that $\Ind_r(G)$ is vertex decomposable for caterpillar graphs. Their approach relied on Woodroofe's result that the independence complex of a chordal hypergraph is shellable~\cite{RW09clutters}. They also constructed chordal graphs whose higher independence complexes are not sequentially Cohen-Macaulay, showing that the behavior of $\Ind_r(G)$ for $r \ge 2$ differs substantially from that of ordinary independence complexes. However, as later observed by Ghosh and Selvaraja~\cite[Remark 4.2]{AG25}, the required chordality assertion for the associated hypergraph fails for arbitrary trees. Consequently, the proof in Abdelmalek et al.~does not establish shellability of $\Ind_r(T)$ for all trees.

These developments left open the question of whether higher independence complexes of trees possess stronger combinatorial properties. This was posed as a conjecture in \cite{PD23chordal}. Unlike the classical independence complex, higher independence complexes allow connected induced subgraphs of bounded size, making the recursive structure of links and deletions substantially more complicated. Moreover, the previously proposed hypergraph approach does not apply to arbitrary trees. The main result of this paper resolves the conjecture completely.

\begin{theorem}[Theorem~\ref{them: V.D for tree}]\label{them: main}
Let $T$ be a tree and let $r \ge 1$. Then the $r$-independence complex $\Ind_r(T)$ is vertex decomposable.
\end{theorem}

Our proof gives a structural characterization of shedding vertices in $\Ind_r(T)$ in terms of rooted subtrees of $T$. We further show that links with respect to connected induced subgraphs inherit compatible shedding vertices, allowing vertex decomposability to be established recursively without relying on hypergraph chordality techniques.

Theorem~\ref{them: main} also has consequences for other graph complexes closely related to higher independence complexes. In particular, path-free complexes form an important class of simplicial complexes associated with graphs and have been studied extensively from combinatorial, algebraic, and topological viewpoints \cite{AKK24,Bi24,AH09,SM11}. Recall that the $t$-path free complex of a graph $G$ is the simplicial complex whose faces are subsets of vertices inducing no path on $t$ vertices. Observe that a subset $A\subseteq V(G)$ induces no path on three vertices if and only if every connected component of $G[A]$ has size at most $2$. Consequently, the $3$-path free complex of $G$ coincides with the $2$-independence complex $\Ind_2(G)$. As an immediate consequence of Theorem~\ref{them: main}, it follows that the $3$-path free complex of a tree is vertex decomposable.

The paper is organized as follows. Section~\ref{sec: 2} recalls background on graphs, simplicial complexes, vertex decomposability, and higher independence complexes. Section~\ref{sec: 3} contains the proof of the main theorem.

\section{Preliminaries}\label{sec: 2}

In this section, we collect the definitions and preliminary results needed for the proof of the main theorem. We begin with basic notions from graph theory and rooted trees, then review simplicial complexes and vertex decomposability, and finally discuss properties of $r$-independence complexes that will be used throughout.

\subsection{Graph theory.}

A (finite, simple) \emph{graph} is a pair $G=(V(G),E(G))$, where $V(G)$ is a finite set of vertices and
\[
E(G)\subseteq \{\{u,v\}:u,v\in V(G),\,u\neq v\}
\]
is the edge set. If $\{u,v\} \in E(G)$, then $u$ and $v$ are said to be \emph{adjacent}, and the edge $\{u,v\}$ is said to be \emph{incident} with both $u$ and $v$.

A graph \(H\) is a \emph{subgraph} of \(G\) if \(V(H)\subseteq V(G)\) and \(E(H)\subseteq E(G)\). For a subset \(A\subseteq V(G)\), the \emph{induced subgraph} \(G[A]\) is the graph with vertex set \(A\) and edge set
\[
E(G[A])
=
\{\{u,v\}\in E(G):u,v\in A\}.
\]

For \(A\subseteq V(G)\), the \emph{deletion} \(G\setminus A\) is the induced subgraph
\[
G[V(G)\setminus A],
\]
obtained by deleting all vertices of \(A\) together with all incident edges.

If \(G_1\) and \(G_2\) have disjoint vertex sets, then their \emph{disjoint union} \(G_1\sqcup G_2\) is the graph with vertex set
\[
V(G_1)\sqcup V(G_2)
\]
and edge set
\[
E(G_1)\sqcup E(G_2).
\]

For \(A\subseteq V(G)\), the \emph{neighbourhood} of \(A\) is
\[
N(A)
=
\{v\in V(G)\setminus A : uv\in E(G)\text{ for some }u\in A\},
\]
and the \emph{closed neighbourhood} of \(A\) is
\[
N[A]=A\cup N(A).
\]

A \emph{path} in \(G\) is a sequence of distinct vertices
\[
v_0,v_1,\dots,v_k
\]
such that \(\{v_i,v_{i+1}\}\in E(G)\) for all \(0\le i\le k-1\). The vertices \(v_0\) and \(v_k\) are called the endpoints of the path. We write \(u-v\) if there exists a path joining \(u\) and \(v\).

A graph is \emph{connected} if every pair of vertices is joined by a path. A \emph{connected component} of \(G\) is a maximal connected subgraph of \(G\).

\subsection{Rooted trees.}

Rooted trees will play an important role in describing the structure of trees locally and in identifying shedding vertices in $r$-independence complexes.

A \emph{tree} $T=(V(T),E(T))$ is a connected graph in which there exists a unique path between every pair of distinct vertices. A \emph{forest} is a graph whose connected components are trees.

A \emph{rooted tree} is a tree $T$ together with a distinguished vertex $v \in V(T)$ called the \emph{root}. Once a root is fixed, the vertices of $T$ inherit a natural hierarchical structure. For each vertex 
\[
u \in V(T)\setminus\{v\},
\]
the \emph{parent} of $u$ is the unique neighbour of $u$ lying on the unique path from $u$ to the root $v$. In this case, $u$ is called a \emph{child} of its parent. The set of children of a vertex $u$ is denoted by $\Ch(u)$. 

A vertex $w$ is called a \emph{descendant} of a vertex $u$ if $u$ lies on the unique path from the root $v$ to $w$. In this case, $u$ is called an \emph{ancestor} of $w$.

We now define the rooted subtree associated with a vertex of a rooted tree.

\begin{definition}\label{def: rooted subtree}
Let $T$ be a tree rooted at $v$. For a vertex $u \in V(T)$, the \emph{rooted subtree} $T_u$ is the subtree of $T$ induced by $u$ together with all descendants of $u$, viewed as a rooted tree with root $u$.
\end{definition}

\begin{figure}[H]
\centering
\begin{tikzpicture}[
every node/.style={circle,fill=black,inner sep=2pt},
level 1/.style={sibling distance=2cm},
level 2/.style={sibling distance=1cm},
level 3/.style={sibling distance=1cm},
level distance=1cm
]

\node[label=above:$v_1$] {}
child { node[label=left:$v_2$] {}
    child { node[label=below:$v_4$] {} }
    child { node[label=below:$v_5$] {} }
}
child { node[label=right:$v_3$] {}
    child { node[label=below:$v_6$] {} }
};

\end{tikzpicture}
\caption{A rooted tree with root $v_1$.}
\label{fig: tree1}
\end{figure}
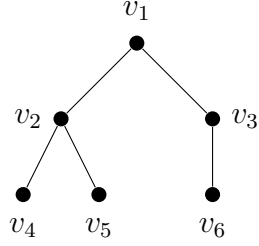

Figure~\ref{fig: tree1} illustrates a rooted tree with root $v_1$. The parent of $v_2$ and $v_3$ is $v_1$, the parent of $v_4$ and $v_5$ is $v_2$, and the parent of $v_6$ is $v_3$. Hence,
\[
\mathrm{Ch}(v_1)=\{v_2,v_3\}, \quad 
\mathrm{Ch}(v_2)=\{v_4,v_5\}, \quad 
\mathrm{Ch}(v_3)=\{v_6\}.
\]

The rooted subtree $T_{v_2}$ is the induced subgraph $T[\{v_2,v_4,v_5\}]$, the rooted subtree $T_{v_3}$ is the induced subgraph $T[\{v_3,v_6\}]$ and the rooted subtree $T_{v_4}$ is the vertex itself. Similarly, the rooted subtrees for the other vertices are defined analogously. 

\subsection{Simplicial complexes.}

An \emph{(abstract) simplicial complex} $\Delta$ on a finite vertex set $V$ is a collection of subsets of $V$ such that if $F \in \Delta$, then every subset of $F$ also belongs to $\Delta$. The elements of $\Delta$ are called \emph{faces}, and the maximal faces with respect to inclusion are called \emph{facets}. We denote the set of facets of $\Delta$ by $\mathcal{F}(\Delta)$. 

The \emph{dimension} of a face $F$ is $\dim F = |F|-1$, where $|F|$ is the cardinality of $F$. The \emph{dimension} of $\Delta$ is the maximum dimension of its faces. An {\it $n$-simplex} is an $n$-dimensional simplicial complex with exactly one facet of dimension $n$. If $\Delta$ is a simplex with vertex set $V$, we write it as $\Delta_V$.

The \emph{empty complex}, denoted by $\{\emptyset\}$, consists only of the empty set and is regarded as the unique $(-1)$-dimensional simplicial complex, while the \emph{void complex}, denoted by $\emptyset$, has no faces. 

A \emph{subcomplex} $\Delta'$ of $\Delta$ is a simplicial complex such that every face of $\Delta'$ is also a face of $\Delta$. For $W \subseteq V$, the \emph{induced subcomplex} of $\Delta$ on $W$ is
\[
\Delta[W]=\{F \in \Delta : F \subseteq W\}.
\]

\medskip

We now recall several standard constructions on simplicial complexes that will be used in the sequel.

\begin{definition}\label{def: join of a SC}
Let $\Delta_1$ and $\Delta_2$ be simplicial complexes on disjoint vertex sets. Their \emph{join}, denoted by $\Delta_1 * \Delta_2$, is the simplicial complex with vertex set $V(\Delta_1)\cup V(\Delta_2)$ and face set
\[
\{F_1 \cup F_2 : F_1 \in \Delta_1,\ F_2 \in \Delta_2\}.
\]
In particular,
\[
\mathcal{F}(\Delta_1 * \Delta_2)
=
\{F_1 \cup F_2 : F_1 \in \mathcal{F}(\Delta_1),\ F_2 \in \mathcal{F}(\Delta_2)\}.
\]
\end{definition}

We next introduce the notions of link and deletion, which are central to recursive definitions such as vertex decomposability.

\begin{definition}\label{def: lk and del}
Let $F$ be a face of a simplicial complex $\Delta$. The \emph{link} of $F$ is the subcomplex of $\Delta$ defined by
\[
\lk_{\Delta}(F)=\{F' \in \Delta \; : \; F \cap F' = \emptyset \; \text{ and}\; F \cup F' \in \Delta\}.
\]

The \emph{deletion} of $F$ is the subcomplex of $\Delta$ defined by
\[
\del_{\Delta}(F)=\{F' \in \Delta \; : \; F \cap F' = \emptyset\}.
\]
\end{definition}

In case $F=\{v\}$, for brevity, we denote $\mathrm{lk}_{\Delta}(\{v\})$ and $\mathrm{del}_{\Delta}(\{v\})$ as $\mathrm{lk}_{\Delta}(v)$ and $\mathrm{del}_{\Delta}(v)$, respectively. 

We next state the link–deletion identities used in the proof of the main theorem.

\begin{lemma}\cite[Lemma 1.1]{PB80}\label{lem: identities}
Let $\Delta$ be a simplicial complex. For faces $F_1, F_2$ of $\Delta$ with $F_1 \cap F_2 = \emptyset$, we have:
\begin{enumerate}
    \item $\del_{\lk_{\Delta}(F_1)}(F_2) = \lk_{\del_{\Delta}(F_2)}(F_1)$,
    \item $\lk_{\lk_{\Delta}(F_1)}(F_2) = \lk_{\Delta}(F_1 \cup F_2)$.
\end{enumerate}
\end{lemma}

\medskip

We now recall the notion of vertex decomposability, which is the main combinatorial property studied in this paper.

\begin{definition}\label{def: VD}
A simplicial complex $\Delta$ is said to be \emph{vertex decomposable} if either $\Delta$ is a simplex (including $\emptyset$ and $\{\emptyset\}$), or there exists a vertex $v$ such that both $\lk_{\Delta}(v)$ and $\del_{\Delta}(v)$ are vertex decomposable and each facet of $\del_{\Delta}(v)$ is a facet of $\Delta$. Such a vertex $v$ is called a \emph{shedding vertex}.
\end{definition}

For a shedding vertex $v$, the facets of the deletion $\del_{\Delta}(v)$ are precisely those facets of $\Delta$ that do not contain $v$. Hence
\[
\mathcal{F}(\del_{\Delta}(v))
=
\{F \in \mathcal{F}(\Delta) : v \notin F\}.
\]
Equivalently, a vertex $v$ is a shedding vertex if for every facet $F$ of $\del_{\Delta}(v)$, the set $F$ is not a face of $\lk_{\Delta}(v)$; that is, the set $F \cup \{v\}$ is not a face of $\Delta$.

\subsection{Higher independence complexes.}

Using the terminology introduced above, we now define the $r$-independence complex of a graph.

\begin{definition}\label{def: r-ind complex}
Let $G$ be a graph and let $r \ge 1$. A subset $A \subseteq V(G)$ is called \emph{$r$-independent} if every connected component of the induced subgraph $G[A]$ contains at most $r$ vertices. The collection of all $r$-independent subsets of $V(G)$ forms a simplicial complex, called the \emph{$r$-independence complex} of $G$, and is denoted by $\Ind_r(G)$.
\end{definition}

Thus, an $r$-independent set allows small connected clusters of size at most $r$, generalizing the classical notion of independence where no edges are allowed.

We now describe how $r$-independence complexes behave under basic graph operations and how their link and deletion structures can be expressed in graph-theoretic terms.

\begin{lemma}\label{lem: disjoint union}
Let $G_1$ and $G_2$ be graphs on disjoint vertex sets. Then
\[
\Ind_r(G_1 \sqcup G_2)=\Ind_r(G_1)*\Ind_r(G_2).
\]
Consequently, $\Ind_r(G_1 \sqcup G_2)$ is vertex decomposable if and only if both $\Ind_r(G_1)$ and $\Ind_r(G_2)$ are vertex decomposable.
\end{lemma}
\begin{proof}
Let $F \subseteq V(G_1)\cup V(G_2)$, and define
\[
F_1 = F \cap V(G_1),
\qquad
F_2 = F \cap V(G_2).
\]
Since $G_1$ and $G_2$ have disjoint vertex sets, the induced subgraph satisfies
\[
(G_1\sqcup G_2)[F]
=
G_1[F_1]\sqcup G_2[F_2].
\]
Hence every connected component of $(G_1\sqcup G_2)[F]$ is a connected component of either $G_1[F_1]$ or $G_2[F_2]$. Therefore, every connected component of $(G_1\sqcup G_2)[F]$ has at most $r$ vertices if and only if the same holds for both $G_1[F_1]$ and $G_2[F_2]$.

It follows that $F$ is an $r$-independent set of $G_1\sqcup G_2$ if and only if $F_1$ and $F_2$ are $r$-independent sets of $G_1$ and $G_2$, respectively. Thus the faces of $\Ind_r(G_1\sqcup G_2)$ are precisely the unions
\[
F_1\cup F_2,
\quad
\text{where }
F_1\in \Ind_r(G_1)
\text{ and }
F_2\in \Ind_r(G_2).
\]
Therefore,
\[
\Ind_r(G_1\sqcup G_2)
=
\Ind_r(G_1)*\Ind_r(G_2).
\]

The final statement follows from \cite[Theorem 3.30]{JJonsson08}, which states that the join of simplicial complexes on disjoint vertex sets is vertex decomposable if and only if both complexes are vertex decomposable.
\end{proof}

Next, we examine how vertex deletion in a graph affects the associated $r$-independence complex.

\begin{lemma}\label{lem: del}
Let $G$ be a graph and let $v \in V(G)$. Then
\[
\del_{\Ind_r(G)}(v)=\Ind_r(G\setminus v).
\]
\end{lemma}

\begin{proof}
By definition,
\[
\del_{\Ind_r(G)}(v)
=
\{F\in \Ind_r(G): v\notin F\}.
\]
Let $F \subseteq V(G)\setminus\{v\}$. Then
\[
G[F]=(G\setminus v)[F].
\]
Hence every connected component of $G[F]$ has at most $r$ vertices if and only if the same holds for $(G\setminus v)[F]$. Therefore,
\[
F\in \Ind_r(G)
\quad\Longleftrightarrow\quad
F\in \Ind_r(G\setminus v).
\]
It follows that
\[
\del_{\Ind_r(G)}(v)=\Ind_r(G\setminus v).
\]
\end{proof}

We now characterize the structure of links in $r$-independence complexes, which will play a key role in inductive arguments later.

\begin{lemma}\label{lem: link}
Let $G$ be a graph and let $A \subseteq V(G)$ such that $|A|=r$ and $G[A]$ is connected. Then
\[
\lk_{\Ind_r(G)}(A)
=
\Ind_r\bigl(G \setminus N_G[A]\bigr),
\]
where $N_G[A]$ denotes the closed neighbourhood of $A$ in $G$.
\end{lemma}

\begin{proof}
Let $F \subseteq V(G)\setminus A$. By definition,
\[
F\in \lk_{\Ind_r(G)}(A)
\quad \Longleftrightarrow \quad
F\cup A \in \Ind_r(G).
\]
Equivalently, every connected component of
\[
G[F\cup A]
\]
has at most $r$ vertices.

Since $G[A]$ is connected and $|A|=r$, the set $A$ forms a connected component of size $r$ in $G[A]$. If some vertex of $F$ were adjacent to a vertex of $A$, then in $G[F\cup A]$ this vertex would belong to the same connected component as $A$, producing a connected component with at least $r+1$ vertices. Hence
\[
F \cap N_G[A]=\varnothing,
\]
and therefore
\[
F \subseteq V(G)\setminus N_G[A].
\]

Conversely, let
\[
F \subseteq V(G)\setminus N_G[A].
\]
Then no vertex of $F$ is adjacent to a vertex of $A$, so
\[
G[F\cup A]
=
G[A]\sqcup G[F].
\]
Since $G[A]$ is connected of size $r$, the condition that every connected component of $G[F\cup A]$ has at most $r$ vertices is equivalent to requiring that every connected component of $G[F]$ has at most $r$ vertices. Therefore,
\[
F\cup A\in \Ind_r(G)
\quad\Longleftrightarrow\quad
F\in \Ind_r(G\setminus N_G[A]).
\]
Hence
\[
\lk_{\Ind_r(G)}(A)
=
\Ind_r(G\setminus N_G[A]).
\]
\end{proof}
These preliminaries will be used in the next section to establish the vertex decomposability of $r$-independence complexes of trees.

\section{Main Theorem}\label{sec: 3}

In this section, we prove that for every integer $r \ge 1$, the $r$-independence complex of a tree is vertex decomposable. The argument is divided into three steps. First, we characterize shedding vertices of $\Ind_r(T)$ in terms of rooted subtrees.
Next, we show that links with respect to connected subgraphs
inherit compatible shedding vertices. Finally, we combine these
results inductively to establish vertex decomposability. 

We first give a structural description of shedding vertices of $\Ind_r(T)$. Informally, the lemma states that a vertex is shedding exactly when enough small rooted subtrees adjacent to it form a connected configuration of size at least $r+1$.

\begin{lemma}\label{lem: all shedding vertices}
Let $T$ be a tree with $|V(T)| \ge r+1$ and $r \ge 1$. A vertex $v \in V(T)$ is a shedding vertex of $\Ind_r(T)$ if and only if, when $T$ is viewed as rooted at $v$, there exists a subset $A \subseteq \mathrm{Ch}(v)$ such that $|V(T_u)| \le r$ for every $u \in A$ and 
\[
\Big|\Big(\bigsqcup_{u \in A} V(T_u)\Big) \cup \{v\}\Big| \ge r+1.
\]
\end{lemma}

\begin{proof}
$(\Leftarrow)$ Assume that for a vertex $v \in V(T)$, when $T$ is viewed as rooted at $v$, there exists a subset $A \subseteq \mathrm{Ch}(v)$ such that $|V(T_u)| \le r$ for every $u \in A$ and 
\[
\Big|\Big(\bigsqcup_{u \in A} V(T_u)\Big) \cup \{v\}\Big| \ge r+1.
\]

After deleting the vertex $v$ from $T$, we obtain
\[
T \setminus \{v\}
=
\bigsqcup_{u \in \mathrm{Ch}(v)} T_u.
\]
Thus, by repeated application of Lemmas~\ref{lem: disjoint union} and ~\ref{lem: del},
\begin{equation}\label{eq: first del}
\del_{\Ind_r(T)}(v)
=
\Ind_r(T \setminus \{v\})
=
\underset{u \in \mathrm{Ch}(v)}{*} \Ind_r(T_u).
\end{equation}

Since $|V(T_u)| \le r$ for each $u \in A$, we have $\Ind_r(T_u) = \Delta_{V(T_u)}$. Let $F$ be a facet of $\del_{\Ind_r(T)}(v)$. Note that if $F$ is not a facet of $\Ind_r(T)$, then $F \cup \{v\}$ must be a face of $\Ind_r(T)$. By the join decomposition in Equation~\eqref{eq: first del}, the facet $F$ contains every vertex of $T_u$ for each $u \in A$. Hence, the induced subgraph $T[F \cup \{v\}]$ contains
\[
T\Big[\Big(\bigsqcup_{u \in A} V(T_u)\Big) \cup \{v\}\Big],
\]
which is connected since $v$ is adjacent to every vertex in $A$. By assumption,
\[
\Big|\Big(\bigsqcup_{u \in A} V(T_u)\Big) \cup \{v\}\Big| \ge r+1.
\]
Therefore, by the definition of the complex, $F \cup \{v\} \notin \Ind_r(T)$.

Hence, every facet of $\del_{\Ind_r(T)}(v)$ is also a facet of $\Ind_r(T)$, and therefore $v$ is a shedding vertex.

\medskip
\noindent
$(\Rightarrow)$ We shall prove this by method of contradiction. Assume there exists a vertex $v \in V(T)$ such that when $T$ is viewed as rooted at $v$, there does not exist a subset $A \subseteq \mathrm{Ch}(v)$ such that $|V(T_u)| \le r$ for every $u \in A$ and 
\[
\Big|\Big(\bigsqcup_{u \in A} V(T_u)\Big) \cup \{v\}\Big| \ge r+1.
\]

We will show that $v$ is not a shedding vertex of $\Ind_r(T)$ by exhibiting a facet of $\del_{\Ind_r(T)}(v)$ that is not a facet of $\Ind_r(T)$.

Partition the set of children of $v$ as
\[
\mathrm{Ch}(v) = A_1 \sqcup A_2,
\]
where $|V(T_u)| \le r$ for each $u \in A_1$ and $|V(T_u)| \ge r+1$ for each $u \in A_2$. By assumption, we have
\begin{equation}\label{eq: assum}
\Big|\Big(\bigsqcup_{u \in A_1} V(T_u)\Big) \cup \{v\}\Big| < r+1.
\end{equation}
Since $|V(T)| \ge r+1$, equation~\eqref{eq: assum} implies that
\[
V(T)\setminus \Big(\Big(\bigsqcup_{u \in A_1} V(T_u)\Big)\cup \{v\}\Big)\neq \emptyset.
\]
Hence there exists a child $u$ of $v$ such that $|V(T_u)| \ge r+1$, and therefore $A_2 \neq \emptyset$.

As before, after deleting $v$ we obtain
\[
T \setminus \{v\}
=
\bigsqcup_{u \in \mathrm{Ch}(v)} T_u,
\]
and hence
\begin{equation}\label{eq: first del 1}
\del_{\Ind_r(T)}(v)
=
\Ind_r(T \setminus \{v\})
=
\underset{u \in \mathrm{Ch}(v)}{*} \Ind_r(T_u).
\end{equation}

Since $|V(T_u)| \le r$ for each $u \in A_1$, we have $\Ind_r(T_u) = \Delta_{V(T_u)}$ for every $u \in A_1$. Thus, every facet $F$ of $\del_{\Ind_r(T)}(v)$ contains the set $\bigsqcup\limits_{u \in A_1} V(T_u)$.

Now, for each $u \in A_2$, since $T_u$ is connected and $|V(T_u)| \ge r+1$, there exists a connected subgraph $C_u$ of $T_u$ of size $r+1$ containing $u$. Indeed, starting with the vertex $u$, one can iteratively add a vertex adjacent to a vertex from the current subgraph until obtaining a connected subgraph on $r+1$ vertices. Then $C_u \setminus \{u\}$ has size $r$. Define
\[
H = \bigsqcup\limits_{u \in A_2} V(C_u \setminus \{u\}).
\]
Since the rooted subtrees $T_u$, $u \in A_2$, are pairwise disjoint, the graphs $C_u \setminus u$ are also pairwise disjoint. Hence, $H$ is a face of $\del_{\Ind_r(T)}(v)$, since the connected components of the induced subgraph $T[H]$ are precisely the disjoint union of graphs $C_u \setminus \{u\}$, each of size exactly $r$, and $v \notin H$.

Let $F$ be a facet of $\del_{\Ind_r(T)}(v)$ with $H \subseteq F$. Then $v \notin F$ and $F \in \Ind_r(T)$. We claim that $A_2 \cap F = \emptyset$. Suppose, for contradiction, that there exists $u \in A_2 \cap F$. Since $u \in A_2$, we know that there exists a connected subgraph $C_u$ of $T$ containing $u$ with $|V(C_u)| = r+1$. Because $H \subseteq F$ and $u \in F$, we have that $C_u$ is a connected subgraph of $T[F]$. Hence $T[F]$ contains a connected subgraph on $r+1$ vertices, contradicting the fact that $F \in \Ind_r(T)$. Therefore, $A_2 \cap F = \emptyset$.

Now, we show that $F \cup \{v\}$ is a face of $\Ind_r(T)$. Observe that every connected component of $T[F \cup \{v\}]$ either lies entirely inside $T[F]$ or is the component containing $v$. Since $F \in \Ind_r(T)$, every connected component of $T[F]$ has size at most $r$. Moreover, as $A_2 \cap F = \emptyset$, the component of $T[F \cup \{v\}]$ containing $v$ is particularly the following component,
\[
T \Big( \Big(\bigsqcup_{u \in A_1} V(T_u)\Big)\cup \{v\} \Big),
\]
whose size is strictly less than $r+1$. Hence this component also has size at most $r$. Therefore every connected component of $T[F \cup \{v\}]$ has size at most $r$, and so $F \cup \{v\} \in \Ind_r(T)$.

Consequently, $F$ is not a facet of $\Ind_r(T)$, and hence $v$ is not a shedding vertex.

\medskip
\noindent
This completes the proof.
\end{proof}

\begin{figure}[H]
\centering
\begin{tikzpicture}[
every node/.style={circle,fill=black,inner sep=2pt},
level 1/.style={sibling distance=2cm},
level 2/.style={sibling distance=1cm},
level 3/.style={sibling distance=1cm},
level 4/.style={sibling distance=1cm},
level distance=1.3cm
]

\node[label=above:$v_1$] {}
child { node[label=above:$v_2$] {}
    child { node[label=below:$v_6$] {} }
}
child { node[label=above:$v_3$] {}
    child { node[label=below:$v_7$] {}
    }
}
child { node[label=above:$v_4$] {}
    child { node[label=right:$v_8$] {} 
    child { node[label=below:$v_9$] {} }
    }
}
child { node[label=above:$v_5$] {}
    child { node[label=right:$v_{10}$] {}
        child { node[label=below:$v_{11}$] {} }
        child { node[label=below:$v_{12}$] {} }
        child { node[label=below:$v_{13}$] {} }
    }
};

\end{tikzpicture}
\caption{}
\label{fig:tree13}
\end{figure}

\begin{example}\label{ex: shedding vertex}
Consider the rooted tree in Figure~\ref{fig:tree13} with root $v_1$ and $r=4$. The children of $v_1$ are
$v_2,v_3,v_4,v_5$, and the corresponding rooted subtrees satisfy
\[
|V(T_{v_2})|=2,\quad |V(T_{v_3})|=2,\quad |V(T_{v_4})|=3,\quad |V(T_{v_5})|=5.
\]
Hence, the subtrees rooted at $v_2,v_3,v_4$ have size at most $r=4$, while $T_{v_5}$ does not.

Taking $A=\{v_2,v_3\}$, we obtain
\[
\Big|\{v_1\}\cup V(T_{v_2})\cup V(T_{v_3})\Big|
= |\{v_1,v_2,v_6,v_3,v_7\}| = 5 = r+1,
\]
so the condition of Lemma~\ref{lem: all shedding vertices} is satisfied. Hence, $v_1$ is a shedding vertex of $\Ind_4(T)$.

The choice of $A$ is not unique; for instance, $A=\{v_3,v_4\}$ also satisfies
\[
|\{v_1\}\cup V(T_{v_2})\cup V(T_{v_4})|
= |\{v_1,v_3,v_7,v_4,v_8,v_9\}| = 6 \ge r+1.
\]

\qed
\end{example}

The previous lemma provides a structural characterization of shedding vertices in $\Ind_r(T)$. In particular, it implies that for every $r \ge 1$ and every tree $T$ with at least $r+1$ vertices, the complex $\Ind_r(T)$ always admits a shedding vertex.

Indeed, fix an arbitrary vertex as a root of $T$ and move downward from the root to the first vertex $v \in V(T)$ whose rooted subtree $T_v$ has size at least $r+1$, while each child’s rooted subtree has size at most $r$, that is,
\[
|V(T_v)| \ge r+1, \quad |V(T_u)| \le r \ \text{for all } u \in \Ch(v).
\]
Viewing $T$ as rooted at $v$, we see that $v$ satisfies the condition in Lemma~\ref{lem: all shedding vertices}, and hence it is a shedding vertex of $\Ind_r(T)$. Therefore, $\Ind_r(T)$ always has a shedding vertex whenever $|V(T)| \ge r+1$.

\medskip

We next show that links with respect to connected subgraphs also admit shedding vertices satisfying a suitable connectivity condition. This will be the key point in the recursive argument for vertex decomposability.

\begin{lemma}\label{lem: existence of shedding vertex in link}
Let $T$ be a tree with $|V(T)| \ge r+1$ and $r \ge 1$. Let $v$ be a shedding vertex of $\Ind_r(T)$. Then, for every connected subgraph $C$ of $T$ containing $v$ and satisfying $|C| \le r-1$, the complex $\lk_{\Ind_r(T)}(V(C))$ has a shedding vertex $w$ such that the graph $C\cup\{w\}$ is connected.
\end{lemma}
\begin{proof}
Let $\Delta = \lk_{\Ind_r(T)}(C)$, and root the tree $T$ at $v$. Throughout, we do not distinguish between a vertex set and the subgraph it induces; in particular, we identify a subgraph with its vertex set when no ambiguity arises.

Since $v$ is a shedding vertex of $\Ind_r(T)$, Lemma~\ref{lem: all shedding vertices} yields a subset $A\subseteq \Ch(v)$ such that $|V(T_u)|\le r$ for every $u\in A$ and
\[
\left|\left(\bigsqcup_{u\in A}V(T_u)\right)\cup\{v\}\right|
\ge r+1.
\]
Set
\[
S=\bigsqcup_{u\in A}V(T_u).
\]
Then $|S|\ge r$ and the induced subgraph $T[S \cup \{v\}]$ is connected of size at least $r+1$. Since $|C|\le r-1$, the set $S\setminus C$ is non-empty. Choose
\[
w\in S\setminus C
\]
so that the rooted subtree \(T_w\) has the largest vertex set among all rooted subtrees \(T_x\) with \(x \in S \setminus C\).

We first show that $\operatorname{parent}(w)\in C$. Suppose $\operatorname{parent}(w) \notin C$. Then
\[
\operatorname{parent}(w)\in S\setminus C.
\]
Since $T_w$ is a proper rooted subtree of $T_{\operatorname{parent}(w)}$, we obtain
\[
|V(T_{\operatorname{parent}(w)})|
>
|V(T_w)|,
\]
contradicting the maximality of $T_w$. Hence $\operatorname{parent}(w)\in C$.

As $C$ is connected and $\operatorname{parent}(w)\in C$, it follows that the subgraph $C\cup\{w\}$ is connected.

We claim that \(w\) is a shedding vertex of \(\Delta\). By definition, it suffices to show that no facet of \(\del_\Delta(w)\) is a face of \(\lk_\Delta(w)\).

Let \(F\) be a facet of \(\del_\Delta(w)\). We will show that that $F\cup\{w\}\notin\Delta$.

We first show
\[
V(T_w)\setminus\{w\}\subseteq F.
\]

Let \(y\) be a descendant of \(w\), i.e., \(y \in V(T_w)\setminus\{w\}\). Suppose, for contradiction, that \(y\notin F\).
We show that
\[
F\cup\{y\}\in\del_\Delta(w),
\]
contradicting the fact that $F$ is a facet of $\del_{\Delta}(w)$.

Since \(y\) is a descendant of \(w\), every unique path from \(y\) to a vertex outside \(T_w\) passes through \(w\). Because $w\notin F\cup C$, it follows that in the induced subgraph
\[
T[F\cup C\cup\{y\}],
\]
the vertex \(y\) lies in a connected component entirely contained in \(T_w \setminus w\).

Moreover,
\[
|V(T_w)|\le r,
\]
because \(w\in S\) and every rooted subtree corresponding to a vertex in \(A\) has at most
\(r\) vertices. Hence every connected component of
\[
T[F\cup C\cup\{y\}]
\]
has at most \(r\) vertices. Therefore
\[
F\cup C\cup\{y\}\in\Ind_r(T),
\]
which implies that $F\cup\{y\}$ is a face of $\Delta$. Since $w\notin F\cup\{y\}$,
we obtain
\[
F\cup\{y\}\in\del_\Delta(w).
\]
But this properly contains the facet \(F\), a contradiction.

Therefore every descendant of \(w\) belongs to \(F\). Equivalently,
\[
V(T_w)\setminus\{w\}\subseteq F.
\]

We now consider two cases.

\medskip
\noindent
\textbf{Case 1.}
Assume that
\[
S\setminus(C\cup T_w)\subseteq F.
\]

Together with
\[
V(T_w)\setminus\{w\}\subseteq F,
\]
this implies that every vertex of \(S\setminus\{w\}\) belongs to \(F\cup C\). Recall, a shedding vertex $v \in C$. Hence
\[
S\cup\{v\}\subseteq F\cup C\cup\{w\}.
\]

Therefore the induced subgraph $T[F\cup C\cup\{w\}]$ contains the connected subgraph $T[S\cup\{v\}]$, which has at least \(r+1\) vertices. Consequently, $F\cup C\cup\{w\}$ is not a face of $\Ind_r(T)$, and hence
\[
F\cup\{w\}\notin\Delta.
\]

\medskip
\noindent
\textbf{Case 2.}
Assume now that there exists
\[
x\in S\setminus(C\cup T_w)
\]
such that $x\notin F$. Since \(F\) is a facet of \(\del_\Delta(w)\), the set $F\cup\{x\}$
cannot be a face of \(\del_\Delta(w)\); otherwise it would properly contain the facet \(F\). Hence
\[
F\cup\{x\}\notin\Delta.
\]
Because
\[
\Delta=\lk_{\Ind_r(T)}(C),
\]
it follows that the set $F\cup C\cup\{x\}$ is not a face of $\Ind_r(T)$.

Therefore the induced subgraph
\[
T[F\cup C\cup\{x\}]
\]
has a connected component \(C'\) satisfying
\[
|V(C')|\ge r+1.
\]

As $F$ is a facet of $\del_{\Delta}(w)$, it follows that
\[
F\cup C\in\Ind_r(T),
\]
every connected component of $T[F\cup C]$ has size at most \(r\). Thus the large connected component \(C'\) appears only after adjoining
the vertex \(x\), and therefore
\[
x\in V(C').
\]

Now $x\in S$,
so there exists \(u\in A\) such that $x\in V(T_u)$.
Since
\[
|V(T_u)|\le r
\]
while
\[
|V(C')|\ge r+1,
\]
the component \(C'\) cannot be a subgraph of \(T_u\). Hence \(C'\) contains a
vertex outside \(T_u\).

Because \(C'\) is connected, the unique path from \(x\) to a vertex outside \(T_u\) lies inside
\(C'\). This path must pass through \(u\), and therefore through the root \(v\). Hence
\[
v\in V(C').
\]

Since \(v \in C\) and \(v \in C'\), and \(C'\) is a connected component of \(T[F \cup C \cup \{x\}]\), the connected subgraph $C$ lies entirely in \(C'\). Hence $C$ is a subgraph of $C'$.

Now consider the graph obtained from \(C'\) by deleting the vertices of $T_x$.
Since every vertex of \(T_x \setminus \{x\}\) is connected to the rest of \(T\) only through \(x\), removing \(V(T_x)\) does not disconnect any path in \(C' \setminus V(T_x)\). Hence
\[
C'' := C' \setminus V(T_x)
\]
is connected.

We claim that \(C \subseteq C''\). Since \(C \subseteq C'\), it suffices to show that \(C \cap V(T_x) = \emptyset\).

Suppose, for contradiction, that there exists a vertex \(y \in C \cap V(T_x)\). Since \(T_x \setminus \{x\}\) is attached to the rest of the tree only through \(x\), any path in \(T\) from \(y\) to a vertex outside \(T_x\) must pass through \(x\). Because \(C\) is connected and contains such a vertex outside \(T_x\) (as \(x \notin C\)), this would force \(x \in C\), contradicting \(x \in S \setminus C\). Hence \(C \cap V(T_x) = \emptyset\).

Therefore no vertex of \(C\) is removed when passing from \(C'\) to \(C'' = C' \setminus V(T_x)\), and thus
\[
C \subseteq C''.
\]

Next we attach the rooted subtree \(T_w\) to $C''$ through the $\operatorname{parent}(w)$. Since $\operatorname{parent}(w)\in C$ which is a subgraph of  $C''$,
the root \(w\) is adjacent to a vertex of \(C''\). Therefore
\[
C''\cup T_w
\]
is connected.

Moreover, by the maximal choice of \(w\),
\[
|V(T_w)|\ge |V(T_x)|.
\]
Hence
\[
|V(C'') \cup V(T_w)|
\ge
|V(C'') \cup V(T_x)| \ge r+1.
\]

Finally, since
\[
V(T_w)\setminus\{w\}\subseteq F,
\]
every vertex of \(T_w\) belongs to \(F\cup\{w\}\). Consequently,
\[
C''\cup T_w
\]
is a connected subgraph of
\[
T[F\cup C\cup\{w\}].
\]

Therefore
\[
T[F\cup C\cup\{w\}]
\]
contains a connected subgraph with at least \(r+1\) vertices. Hence
\[
F\cup C\cup\{w\}\notin\Ind_r(T),
\]
and thus
\[
F\cup\{w\}\notin\Delta.
\]

Thus no facet of \(\del_\Delta(w)\) is a facet of \(\lk_\Delta(w)\). Therefore \(w\) is a
shedding vertex of \(\Delta\).
\end{proof}

We illustrate the construction of the shedding vertex in Lemma~\ref{lem: existence of shedding vertex in link} with the following example.

\begin{example}
Consider the rooted tree in Figure~\ref{fig:tree13} with root $v_1$ and $r=4$. From Example~\ref{ex: shedding vertex}, we know that $v_1$ is a shedding vertex of $\Ind_4(T)$. 

Let 
\[
C = T[\{v_1,v_4\}]
\]
be a connected subgraph of $T$ containing the shedding vertex $v_1$. Again from Example~\ref{ex: shedding vertex}, we may take
\[
A=\{v_3,v_4\},
\]
which satisfies the conditions of Lemma~\ref{lem: all shedding vertices}. Hence,
\[
V(T_{v_3}) \cup V(T_{v_4}) = \{v_3,v_7\} \cup \{v_4,v_8,v_9\}
= \{v_3,v_4,v_7,v_8,v_9\}.
\]

Now consider the set
\[
\big(V(T_{v_3}) \cup V(T_{v_4})\big)\setminus C
= \{v_3,v_7,v_8,v_9\}.
\]
Among these vertices, we compare the sizes of their rooted subtrees:
\[
|V(T_{v_3})|=2,\quad |V(T_{v_8})|=2,\quad |V(T_{v_7})|=1,\quad |V(T_{v_9})|=1.
\]
Since $v_3$ and $v_8$ has the largest rooted subtree among these candidates, one may choose $w=v_3$ or $w=v_8$. WLOG, choose $w = v_8$. Since, the parent of $v_8$ is $v_4$ which is in $C$, $C\cup\{v_8\}$ is connected in $T$.

We now show that $v_8$ is a shedding vertex of $\Delta=\lk_{\Ind_4(T)}(C)$. Let $F$ be a facet of $\del_{\lk_{\Ind_4(T)}(C)}(v_8)$. By definition, $v_8 \notin F$ and $F \cup C \in \lk_{\Ind_4(T)}(C)$.

Since $|V(T_{v_8})|=2 \le r=4$, the subtree rooted at $v_8$ consists only of $v_8$ and its descendant $v_9$. As $v_8 \notin F$, every connected subgraph containing $v_9$ and a vertex outside $T_{v_8}$ must have a vertex $v_8$. Thus every facet of $\del_{\lk_{\Ind_r(T)}(C)}(v_8)$ contains $v_9$. 

Suppose $v_3 \notin F$. Since $F$ is a facet of $\del_{\lk_{\Ind_4(T)}(C)}(v_8)$,
\[
F \cup \{v_1,v_3,v_4\} \notin \Ind_4(T).
\]
Since
\[
F \cup \{v_1,v_4\} \in \Ind_4(T),
\]
there exists a connected component of size at least $5$ in
\[
T[F \cup \{v_1,v_3,v_4\}]
\]
containing $v_3$.

Suppose $v_2,v_7 \in F$. Then
\[
T[\{v_1,v_2,v_3,v_4,v_7\}]
\]
is connected. Since $v_9 \in F$, replacing $v_3,v_7$ by $v_8,v_9$ gives the connected induced subgraph
\[
T[\{v_1,v_2,v_4,v_8,v_9\}],
\]
which has size $5$. Hence
\[
F \cup \{v_1,v_4,v_8\} \notin \Ind_4(T).
\]
Therefore
\[
F \cup \{v_8\} \notin \lk_{\Ind_4(T)}(C).
\]

This example illustrates the main idea behind Lemma~\ref{lem: existence of shedding vertex in link}: a maximal rooted subtree outside $C$ can replace another rooted subtree while preserving a forbidden connected component.

\qed
\end{example} 
The existence of shedding vertices established above allows us to prove that links with respect to connected subgraphs are themselves vertex decomposable.

\begin{lemma}\label{lem: shedding vertex for the particular link}
Let $T$ be a tree with $|V(T)| \ge r+1$ and let $r \ge 1$. Assume that $\Ind_r(T')$ is vertex decomposable for every tree $T'$ with $|V(T')|<|V(T)|$. Let $v$ be a shedding vertex of $\Ind_r(T)$. Then, for every connected subgraph $C$ of $T$ containing $v$ with $|C|\le r$, the complex
\[
\lk_{\Ind_r(T)}(C)
\]
is vertex decomposable.
\end{lemma}

\begin{proof}
Let
\[
\Delta=\lk_{\Ind_r(T)}(C).
\]

We proceed by induction on $|V(T)|-|C|$.

If $|V(T)|-|C|=1$, then $|V(T)|=r+1$ and $|C|=r$. Hence,
\[
\Delta=\{\emptyset\},
\]
which is vertex decomposable.

Assume that for every connected subgraph $C'$ of $T$ containing $v$ with
\[
|V(T)|-|C'|<|V(T)|-|C|,
\]
the complex $\lk_{\Ind_r(T)}(C')$ is vertex decomposable.

First suppose that $|C|=r$. By Lemma~\ref{lem: link},
\[
\Delta=\Ind_r(T\setminus N_T[C]).
\]
Since $T\setminus N_T[C]$ is a forest whose connected components have fewer vertices than $T$, the hypothesis implies that the $r$-independence complex of each connected component is vertex decomposable. Therefore, by Lemma~\ref{lem: disjoint union}, $\Delta$ is vertex decomposable.

Hence we may assume that $|C|\le r-1$. By Lemma~\ref{lem: existence of shedding vertex in link}, the complex $\Delta$ has a shedding vertex $w$ such that $C\cup\{w\}$ is connected.

By Lemma~\ref{lem: identities} and Lemma~\ref{lem: del},
\[
\del_\Delta(w)
=
\del_{\lk_{\Ind_r(T)}(C)}(w)
=
\lk_{\del_{\Ind_r(T)}(w)}(C)
=
\lk_{\Ind_r(T\setminus w)}(C).
\]

Since $T\setminus w$ is a forest whose connected components each have fewer vertices than $T$, the hypothesis implies that the $r$-independence complex of every connected component of $T\setminus w$ is vertex decomposable. Therefore, by Lemma~\ref{lem: disjoint union}, $\Ind_r(T\setminus w)$ is vertex decomposable.

By \cite[Theorem 3.30]{JJonsson08}, the link of a face in a vertex decomposable complex is again vertex decomposable. Hence,
\[
\lk_{\Ind_r(T\setminus w)}(C)
\]
is vertex decomposable. Consequently, $\del_\Delta(w)$ is vertex decomposable.

Now, by Lemma~\ref{lem: identities},
\[
\lk_\Delta(w)
=
\lk_{\Ind_r(T)}(C\cup\{w\}).
\]
Since $C\cup\{w\}$ is connected and
\[
|V(T)|-|C\cup\{w\}|
<
|V(T)|-|C|,
\]
the induction hypothesis implies that $\lk_\Delta(w)$ is vertex decomposable.

Thus both $\del_\Delta(w)$ and $\lk_\Delta(w)$ are vertex decomposable. Since $w$ is a shedding vertex of $\Delta$, it follows that $\Delta$ is vertex decomposable.
\end{proof}

We are now in a position to combine the previous results and complete the proof of the main theorem.

\begin{theorem}\label{them: V.D for tree}
Let $T$ be a tree and $r \ge 1$. Then $\Ind_r(T)$ is vertex decomposable.
\end{theorem}

\begin{proof}
We proceed by induction on $n = |V(T)|$.

If $n \le r$, then every subset of the vertex set of $T$ is $r$-independent, so $\Ind_r(T)$ is a simplex and hence vertex decomposable.

Assume the result holds for all trees with fewer than $n$ vertices, and let $T$ be a tree with $|V(T)| = n$.

By Lemma~\ref{lem: all shedding vertices}, there exists a shedding vertex $v \in V(T)$. To show that $\Ind_r(T)$ is vertex decomposable, it suffices to verify that both the deletion and the link with respect to $v$ are vertex decomposable.

\medskip
\noindent
\textbf{Deletion:}
By Lemma~\ref{lem: del}, we have
\[
\del_{\Ind_r(T)}(v) = \Ind_r(T \setminus v).
\]

Since $T \setminus v$ is a forest whose components are trees with fewer vertices than $T$, the induction hypothesis implies that the $r$-independence complex of each component is vertex decomposable. Thus, by Lemma~\ref{lem: disjoint union}, $\Ind_r(T \setminus v)$ is vertex decomposable.

\medskip
\noindent
\textbf{Link:} Since $v$ is a connected induced subgraph of $T$ containing $v$, the induction hypothesis and Lemma~\ref{lem: shedding vertex for the particular link} imply that
\[
\lk_{\Ind_r(T)}(v)
\]
is vertex decomposable.

\medskip
\noindent
As both $\del_{\Ind_r(T)}(v)$ and $\lk_{\Ind_r(T)}(v)$ are vertex decomposable and $v$ is a shedding vertex, $\Ind_r(T)$ is vertex decomposable.
\end{proof}

We conclude the article with a question. 
Lemma~\ref{lem: all shedding vertices} completely characterizes shedding vertices in $\Ind_r(T)$. 
Hence, it is natural to ask for an analogous description for links.

\begin{question}
Let $T$ be a tree, $r \ge 1$, and let $A \subseteq V(T)$ such that each connected component of the induced subgraph $T[A]$ has size at most $r$. Is there a combinatorial characterization of the shedding vertices of 
$\lk_{\Ind_r(T)}(A)$?
\end{question}

\section*{Acknowledgements}

The author thanks Priyavrat Deshpande for valuable suggestions on an earlier draft of this article. The author also acknowledges partial support from an Infosys Foundation grant.

\nocite{*}
\bibliographystyle{abbrv}
\bibliography{references}

\end{document}